\documentclass[12pt]{amsart}
\usepackage{amssymb}
\usepackage{amsmath}
\usepackage{amsthm}

\DeclareMathOperator{\ad}{ad}

\newcommand{\e}{\varepsilon}

\newcommand{\I}{{\mathcal I}}
\newcommand{\K}{{\mathcal K}}

\newcommand{\N}{{\mathcal N}}

\newcommand{\p}{\partial}
\newcommand{\s}{{\rm Symb}}
\renewcommand{\S}{\Sigma}
\newcommand{\T}{{T^\ast M}}

\newtheorem{lemma}{Lemma}
\newtheorem{proposition}{Proposition}
\newtheorem{theorem}{Theorem}

\begin{document}
\title[Deformations of a formal symplectic groupoid]
{Infinitesimal deformations of a formal symplectic groupoid}
\author[Alexander Karabegov]{Alexander Karabegov}
\address[Alexander Karabegov]{Department of Mathematics, Abilene Christian University, ACU Box 28012, Abilene, TX 79699-8012}
\email{axk02d@acu.edu}

\begin{abstract}
Given a formal symplectic groupoid $G$ over a Poisson
manifold $(M,\pi_0)$, we define a new object, an infinitesimal deformation
of $G$, which can be thought of as a formal symplectic groupoid
over the manifold $M$ equipped with an infinitesimal deformation
$\pi_0 + \e \pi_1$ of the Poisson bivector field $\pi_0$. The source and target
mappings of a deformation of $G$ are deformations of the source and
target mappings of $G$. To any pair of natural star products $(\ast,\tilde\ast)$
having the same formal symplectic groupoid $G$ we relate an infinitesimal
deformation of $G$. We call it the deformation groupoid of
the pair $(\ast,\tilde\ast)$. We give explicit formulas for the source and target
mappings of the deformation groupoid of a pair of star products
with separation of variables on a K\"ahler-Poisson manifold. Finally,
we give an algorithm for calculating the principal symbols of the
components of the logarithm of a formal Berezin transform of a
star product with separation of variables. This algorithm is based
upon some deformation groupoid.
\end{abstract}

\date{June 14, 2010}
\maketitle

\section{Introduction}

Symplectic groupoids introduced in \cite{Ka}, \cite{W}, and \cite{Z}, are heuristic semiclassical counterparts of associative algebras treated as quantum objects (see \cite{CDF}). Passing to formal objects allows to make this correspondence precise. Namely, to each (natural) formal star algebra one can relate a formal symplectic groupoid (\cite{CMP3}).

In \cite{CDF} a formal symplectic groupoid was defined in terms of a formal generating function of the Lagrangian product space. This formal generating function was then calculated for the formal symplectic groupoid of the Kontsevich star product (see \cite{K}). 

In \cite {CMP3} a formal symplectic groupoid was defined in terms of the formal neighborhood of the Lagrangian unit space in an ambient symplectic manifold. It was shown that to each natural star product on a Poisson manifold $M$ one can relate a formal symplectic groupoid on the formal neighborhood of the zero section of the cotangent bundle $\T$. In~ \cite{LMP2} the source and target mappings of the formal symplectic groupoid of Fedosov's star product were calculated  (see \cite{F1}). In \cite{CMP3} it was shown that the formal symplectic groupoid of a star product with separation of variables can be characterized by the ``separation of variables" property and that for any K\"ahler-Poisson manifold $M$ there exists a unique formal symplectic groupoid with separation of variables over ~$M$.

Given a Poisson manifold $M$ equipped with a Poisson bivector field $\pi_0$, a pair of natural star products $(\ast,\tilde \ast)$ on $(M,\pi_0)$ having the same formal symplectic groupoid $G$ determine an infinitesimal deformation $\pi_0 + \e \pi_1$ of the Poisson bivector field $\pi_0$. It turns out that one can relate to the pair $(\ast,\tilde \ast)$ a ``deformation" formal symplectic groupoid over the manifold $(M, \pi_0 + \e \pi_1)$. Its source and target mappings  are infinitesimal deformations of the source and target mappings of  $G$. 

Given a K\"ahler-Poisson manifold $M$ and an infinitesimal deformation $\pi_0 + \e \pi_1$ of the K\"ahler-Poisson bivector field $\pi_0$, there exists a unique deformation groupoid with separation of variables over $(M, \pi_0 + \e \pi_1)$. We give explicit formulas of its source and target mappings. 

It was shown in \cite{CMP3} that if the formal Berezin transform $B$ of a star product with separation of variables on a K\"ahler-Poisson manifold $M$ is written in the exponential form, $B = \exp \frac{1}{\nu}X$, with $X = \nu^2 X_2 + \nu^3 X_3 + \ldots$, then the operators $X_{2k}$ and $X_{2k+1}$ are of order not greater than $2k$. In \cite{CMP3} we gave an algorithm of calculating the principal symbol of order $2k$ of the operator $X_{2k}$ for all $k \geq 1$ in terms of the formal symplectic groupoid with separation of variables over $M$. In this paper we give a calculation algorithm of the principal symbol of order $2k$ of the operator $X_{2k+1}$ for all $k \geq 1$ in terms of some deformation groupoid with separation of variables over $M$. It was the main motivation of this paper to find this algorithm.

{\bf Acknowledgments.} I am very grateful to Professor Daniel Sternheimer and to the referees for the suggestions which helped me to make important changes to this paper.

\section{Deformation formal symplectic groupoids}

Given a manifold $X$ with a submanifold $Y \subset X$, denote by $I_Y$ the vanishing ideal of $Y$ in $C^\infty (X)$. Then the quotient algebra
\[
     C^\infty(X,Y) : = C^\infty(X)/(\cap_{l=1}^\infty I_Y^l)
\]
can be thought of as the algebra of smooth functions on the formal neighborhood $(X,Y)$ of the submanifold $Y$ in $X$.

A formal symplectic groupoid over a Poisson manifold $(M,\pi)$ is given by the following collection of data: a symplectic manifold $\S$ with a Lagrangian submanifold $\Lambda$ which is a copy of $M$, the identification inclusion $\epsilon: M \to \S$ with $\epsilon(M) = \Lambda$, its dual $E: C^\infty(\S,\Lambda) \to C^\infty(M)$, a Poisson bracket 
$\{\cdot,\cdot\}_\S$ on $C^\infty(\S,\Lambda)$ induced by the symplectic structure on $\S$, source and target mappings
\[
     S,T: C^\infty(M) \to C^\infty(\S,\Lambda),
\] 
an inverse mapping, and a comultiplication mapping satisfying a number of axioms. 
In particular, $ESf = f, ETf = f$ for any $f \in C^\infty(M)$, $S$ is a Poisson mapping, $T$ is an anti-Poisson mapping, and for any $f,g\in C^\infty(M)$ the elements $Sf$ and $Tg$ Poisson commute with respect to the Poisson bracket $\{\cdot,\cdot\}_\S$. The details can be found in \cite{CMP3}. A formal symplectic groupoid is completely determined by its source mapping. Namely, given manifolds $M, \Sigma, \Lambda$ as above, the dual inclusion mapping $E: C^\infty(\S,\Lambda) \to C^\infty(M)$, and a Poisson mapping $S:C^\infty(M) \to C^\infty(\S,\Lambda)$ such that $ESf = f$ for any $f \in C^\infty(M)$, then there exists a unique formal symplectic groupoid on $(\S,\Lambda)$ for which $S$ is the source mapping.

Now let $(M,\pi_0)$ be a Poisson manifold and $\pi_1$ be a bivector field on $M$ such that the Schouten bracket of $\pi_0$ and $\pi_1$ vanishes, $[\pi_0,\pi_1]=0$, i.e., $\pi_1$ is a Poisson 2-cocycle. Denote by $\e$ a formal nilpotent parameter such that $\e^2 =0$. Then $[\pi_0 + \e \pi_1, \pi_0 + \e \pi_1]=0$, which means that $\pi_0 + \e \pi_1$ can be thought of as an infinitesimal deformation of the Poisson tensor $\pi_0$. One can define a formal symplectic groupoid over the manifold $(M, \pi_0 + \e \pi_1)$ using the same approach as in \cite{CMP3}. We will call it a deformation formal symplectic groupoid over $(M, \pi_0 + \e \pi_1)$. We will specify the definitions of the source and target mappings of a deformation formal symplectic groupoid. 

Let, as above, $\S$ be a symplectic manifold with a Lagrangian submanifold $\Lambda \subset\S$ identified with $M$ via an inclusion $\epsilon: M \to \S$ such that $\epsilon(M) = \Lambda$, and $E: C^\infty(\S,\Lambda) \to C^\infty(M)$ be the dual mapping of $\epsilon$.  We equip $C^\infty(M)[\e]$ with the deformed Poisson bracket $\{\cdot,\cdot\} = \{\cdot,\cdot\}_0 + \e \{\cdot,\cdot\}_1$ corresponding to $\pi_0 + \e \pi_1$ and extend the Poisson bracket $\{\cdot,\cdot\}_{\S}$ to $C^\infty(\S, \Lambda)[\e]$ by $\e$-linearity. The source and target  mappings of a deformation formal symplectic groupoid over $(M, \pi_0 + \e \pi_1)$ are defined as $\e$-linear mappings
\[
   S, T: C^\infty(M)[\e] \to C^\infty(\S, \Lambda)[\e],
\]
which are a Poisson and an anti-Poisson morphism of Poisson algebras, respectively: for any $f,g \in C^\infty(M)[\e]$,
\begin{align}\label{E:deformed}
   E(Sf) = f, \ E(Tf) = f, \ S(fg) = Sf\ Sg, \ T(fg) = Tf\ Tg,  \\
S(\{f,g\}) = \{Sf, Sg\}_{\S},
  T(\{f,g\}) =-\{Tf, Tg\}_{\S}, \mbox{ and } \{Sf, Tg\}_{\S}=0. \nonumber
\end{align}
The mappings $S$ and $T$ can be written in components: for $f \in C^\infty(M)$,
\[
   Sf = S_0 f + \e S_1 f \mbox{ and } Tf = T_0 f + \e T_1 f.
\]
It follows from equations (\ref{E:deformed}) that $S_0$ and $T_0$ satisfy, respectively, the axioms of the source and target mappings of a formal symplectic groupoid over 
$(M, \pi_0)$. For $f,g \in C^\infty(M)$ we have three groups of equations on the components $S_i,T_i, i = 0,1$:
\begin{equation}\label{E:unit}
    E(S_1 f) =0,\ E(T_1 f) =0;
\end{equation}
\begin{align}\label{E:product}
S_1(fg) = S_0f \, S_1g + S_1 f\, S_0 g, \\
T_1(fg) = T_0f\, T_1g + T_1 f\, T_0 g; \nonumber
\end{align}
and
\begin{align}\label{E:bracket}
  S_1\{f,g\}_0 + S_0\{f,g\}_1 = \{S_0 f,S_1 g\}_{\S} + \{S_1 f,S_0 g\}_{\S},\nonumber \\
T_1\{f,g\}_0 + T_0\{f,g\}_1 = -\{T_0 f,T_1 g\}_{\S} - \{T_1 f,T_0 g\}_{\S}, \\
\{S_0 f, T_1 g\}_{\S} + \{S_1 f, T_0 g\}_{\S}=0.\nonumber
\end{align}
It was shown in \cite{CMP3} that to any natural star product on a Poisson manifold $M$ one can relate a formal symplectic groupoid over $M$. In the next section we will show that, given two natural star products on a Poisson manifold $(M,\pi_0)$ with the same formal symplectic groupoid $G$, one can construct a deformation formal symplectic groupoid over $(M, \pi_0 + \e \pi_1)$, where $\pi_1$ is determined by the pair of star products. The zero components $S_0$ and $T_0$ of the source and target mappings of the deformation groupoid are the source and target mappings of the formal symplectic groupoid $G$, respectively.

\section{Deformation formal symplectic groupoid of a pair of star products}

In what follows $\nu$ denotes a formal parameter. A formal differential operator $A = A_0 + \nu A_1 + \ldots$ on a manifold $M$ is called natural if the order of the differential operator $A_r$ is not greater than $r$ for all $r \geq 0$ (see \cite{LMP2}). The natural formal differential operators on $M$ form an algebra which we denote by $\N$.
The principal symbol $\s_r(A_r)$ of $A_r$ is a homogeneous fibrewise polynomial function of order $r$ on $\T$.  The formal series
\[
   \sigma(A) = \sum_{r = 0}^\infty \s_r(A_r)
\]
can be treated as an element of $C^\infty(\T,Z)$, where $Z$ is the zero section of $\T$. We call $\sigma(A)$ the $\sigma$-symbol of the natural formal differential operator $A$. The mapping
\[
    \sigma: \N \to C^\infty(\T,Z)
\]
is a homomorphism whose kernel is $\nu \N$. Moreover, for any $A,B \in \N$,
\[
     \sigma \left(\frac{1}{\nu} [A,B]\right) = \{\sigma(A),\sigma(B)\}_{\T},
\]
where $\{\cdot,\cdot\}_{\T}$ is the Poisson bracket corresponding to the standard symplectic structure on the cotangent bundle $\T$.

A formal differential star product $\ast$ on a Poisson manifold $(M,\pi)$ (see \cite{BFFLS}) is an associative product on the space of formal functions $C^\infty(M)[[\nu]]$ given by the formula
\[
     f \ast g = \sum_{r=0}^\infty \nu^r C_r(f,g).
\]
The operators $C_r$ are bidifferential, $C_0(f,g) = fg$, and $C_1(f,g) - C_1(g,f) = \{f,g\}$, where $\{\cdot,\cdot\}$ is the Poisson bracket given by the bivector field $\pi$. The unit constant 1 is the unity of the star product. It was proved by Kontsevich in \cite{K} that star products exist on arbitrary Poisson manifolds.

A differential star product can be restricted to an open set. Thus one can consider star products of local functions. A star product is called natural (see \cite{GR}) if the bidifferential operator $C_r$ is of order not greater than $r$ in each argument. Star products of Fedosov and Kontsevich and star products with separation of variables are natural (see \cite{GR},\cite{BW},\cite{LMP2}, and \cite{CMP3}). Given $f,g \in C^\infty(M) [[\nu]]$, denote by $L_f$ and $R_g$ the left star multiplication operator by $f$ and the right star multiplication operator by $g$, respectively, so that $L_f g = f \ast g = R_g f$. The associativity of $\ast$ is equivalent to the fact that $[L_f,R_g]=0$ for any $f,g$. A star product $\ast$ is natural if and only if the operators $L_f, R_f$  are natural for any $f \in C^\infty(M) [[\nu]]$. 

In \cite{CMP3} we introduced a formal symplectic groupoid of a natural formal deformation quantization on a Poisson manifold $M$. This groupoid is defined on the formal neighborhood $(\T,Z)$ of the zero section of the cotangent bundle of $M$. The source and target mappings of this groupoid are defined as follows: for $f\in C^\infty(M)$,
\[
    Sf = \sigma(L_f), \ Tf = \sigma(R_f). 
\]

{\it Remark.} Observe that the expression $\sigma(L_f)$ used to define the source mapping of the formal symplectic groupoid of a deformation quantization makes sense for any formal function $f = f_0 + \nu f_1 + \ldots$. However, this expression depends only on the component $f_0$, since the operator $L_{\nu^r f_r} \in \nu^r N$ is in the kernel of the mapping $\sigma$ for any $r \geq 1$. Same is true for the expression used to define the target mapping.

Now let $\ast$ and $\tilde \ast$ be two natural star products on a Poisson manifold $(M,\pi_0)$ defined by bidifferential operators $\{C_r\}$ and $\{\tilde C_r\}$, respectively. The left and right star multiplication operators by a formal function $f$ for these star products will be denoted $L_f, R_f$ and $\tilde L_f,\tilde R_f$, respectively. Assume that these star products have the same formal symplectic groupoid with the source and target mappings denoted by $S_0$ and $T_0$, respectively, so that
\begin{equation}\label{E:szerotzero}
   S_0f = \sigma(L_f) =  \sigma(\tilde L_f) \mbox{ and } T_0f = \sigma(R_f) =  \sigma(\tilde R_f).
\end{equation}

{\it Remark.} The condition that two natural star products have the same formal symplectic groupoid is not very restrictive. It can be derived from the results of \cite{CMP3} that for any two natural star products $\ast_1$ and $\ast_2$ on a Poisson manifold $M$ one can find a natural star product $\ast$ on $M$ which is equivalent to the star product $\ast_2$ and has the same formal symplectic groupoid as $\ast_1$.

\begin{lemma}\label{L:equal}
For the natural star products $\ast$ and $\tilde\ast$ with the same formal symplectic groupoid the bidifferential operators $C_1$ and $\tilde C_1$ coincide.
\end{lemma}
\begin{proof}
Fix an arbitrary function $f \in C^\infty (M)$ and consider the operators $Ag = C_1(f,g)$ and $\tilde Ag = \tilde C_1(f,g)$ on $C^\infty(M)$. Since the star products $\ast$ and $\tilde\ast$ are natural, the operators $A$ and $\tilde A$ are of order not greater than one. Since the unit constant 1 is the unity for both products, these operators annihilate constants. Their principal symbols of order one are both equal to the homogeneous component of order one of the source mapping $S_0f$. Therefore these operators are equal, which implies that $C_1 = \tilde C_1$. 
\end{proof}
In what follows we will use the notation 
\[
     D_r(f,g) := C_r(f,g) - C_r(g,f) \mbox{ and } \tilde D_r(f,g) := \tilde C_r(f,g) - \tilde C_r(g,f)
\]
for $r \geq 1$. Denote by $\{\cdot,\cdot\}_0$ the Poisson bracket on $M$ given by the Poisson tensor $\pi_0$. Then $D_1(f,g) = \tilde D_1(f,g) = \{f,g\}_0$. Since $C_1 = \tilde C_1$, it can be checked directly that the skew-symmetric bidifferential operator 
\[
   \{f,g\}_1 := D_2(f,g)  - \tilde D_2(f,g)
\]
is a derivation in each argument and is determined by a bivector field $\pi_1$ which is a Poisson 2-cocycle with respect to the Poisson bivector field $\pi_0$.  Our goal is to construct  a deformation formal symplectic groupoid over  $(M, \pi_0 + \e \pi_1)$ from the star products $\ast$ and $\tilde \ast$.

We define $\e$-linear mappings
\[
     S,T: C^\infty(M)[\e] \to C^\infty(\T,Z)[\e]
\] 
written in components as $S = S_0 + \e S_1$ and $T = T_0 + \e T_1$, as follows. We take  $S_0$ and $T_0$ as in (\ref{E:szerotzero}) and set
\begin{equation}\label{E:sonetonestar}
     S_1f = \sigma\left( \frac{1}{\nu}(L_f - \tilde L_f)\right) \mbox{ and } T_1f = \sigma\left( \frac{1}{\nu}(R_f - \tilde R_f)\right)
\end{equation}
for $f \in C^\infty(M)$. We will show that the mappings $S$ and $T$ satisfy the axioms of a source and a target mapping, respectively. This will imply the existence of a unique deformation formal symplectic groupoid over $(M, \pi_0 + \e \pi_1)$ with the source mapping $S$ and the target mapping $T$.

{\it Remark.} Since the formal differential operators $L_f$ and $\tilde L_f$ are natural and have the same $\sigma$-symbol, we see that $L_f - \tilde L_f \in \nu \N$, which justifies that $S_1$ is well defined. Same argument shows that the mapping $T_1$ is also well defined.

{\it Remark.} It is easy to check that if the function $f$ in the expressions defining the mappings $S_1$ and $T_1$ is formal, $f = f_0 + \nu f_1 + \ldots$, then these expressions depend only on $f_0$.

\begin{proposition}
The mappings $S_i, T_i, i =0,1$, related to the star products $\ast$ and $\tilde\ast$ satisfy axioms (\ref{E:unit}), (\ref{E:product}), and (\ref{E:bracket}).
\end{proposition}
\begin{proof}
First we will check the axioms (\ref{E:unit}) for the mappings $S_1$ and $T_1$.  Given a natural operator $A = A_0 + \nu A_1 + \ldots \in \N$, the operator $A_0$ is the pointwise multiplication operator by a function. We see that $E \sigma(A) = A_0$ is the coefficient at the zeroth degree of $\nu$ of the formal function $A1$. For
\[
     A = \frac{1}{\nu}(L_f - \tilde L_f),
\]
where $f \in C^\infty(M)$, we have
\[
    A1 = \frac{1}{\nu} (f \ast 1 - f \tilde \ast 1) =0,
\]
 whence $E(S_1f) = E(\sigma(A)) =0$. Similarly, $E(T_1 f) =0$.

Then we will check the first axiom in (\ref{E:product}).  For $f,g \in C^\infty(M)$ we have
\begin{align}\label{E:prelim1}
   \sigma\left( \frac{1}{\nu}(L_{f\ast g} - \tilde L_{f\tilde\ast g})\right) = \sigma\left( \frac{1}{\nu}(L_f L_g - \tilde L_f \tilde L_g)\right) =\nonumber  \\
\sigma\left( \frac{1}{\nu}(L_f L_g - L_f \tilde L_g) + \frac{1}{\nu}(L_f \tilde L_g - \tilde L_f \tilde L_g)\right) = \\
\sigma\left( L_f \left(\frac{1}{\nu}(L_g - \tilde L_g)\right)\right) + \sigma\left(\frac{1}{\nu}(L_f - \tilde L_f) \tilde L_g \right) = \nonumber\\
S_0 f \, S_1 g + S_1 f \, S_0 g. \nonumber
\end{align}
Next,
\begin{align}\label{E:prelim2}
\sigma\left(\frac{1}{\nu}(L_{f\ast g} - L_{fg})\right) = \sigma \left(L_{C_1(f,g)}\right) = 
S_0 (C_1(f,g)).
\end{align}
Similarly,
\begin{equation}\label{E:prelim3}
 \sigma\left(\frac{1}{\nu}(\tilde L_{f\tilde \ast g} - \tilde L_{fg})\right) = S_0 (\tilde C_1(f,g)).
\end{equation}
Using equations (\ref{E:prelim1} - \ref{E:prelim3}) and the fact that  $C_1 = \tilde C_1$ proved in Lemma \ref{L:equal}, we obtain the first axiom in (\ref{E:product}):
\begin{align*}
 & S_1(fg)  = \sigma\left(\frac{1}{\nu}(L_{fg} - \tilde L_{fg}) \right) =\\
 & \sigma\left( \frac{1}{\nu}(L_{f\ast g} - \tilde L_{f\tilde\ast g})\right) = 
 S_0 f \, S_1 g + S_1 f \, S_0 g.
\end{align*}
The second axiom can be proved along the same lines.

To prove the first axiom in (\ref{E:bracket}) we observe that 
\[
     [f,g]_\ast := f \ast g - g \ast f = \nu \{f,g\}_0 + \nu^2 D_2(f,g) \pmod{\nu^3}.
\]
It follows that
\begin{equation}\label{E:ver1}
  \sigma\left(L_{\frac{1}{\nu^2} [f,g]_\ast} - L_{\frac{1}{\nu}\{f,g\}_0}\right) = \sigma\left (L_{D_2(f,g)}\right) = S_0 D_2(f,g).
\end{equation}
Similarly,
\begin{equation}\label{E:ver2}
  \sigma\left(\tilde L_{\frac{1}{\nu^2} [f,g]_{\tilde\ast}} - \tilde L_{\frac{1}{\nu}\{f,g\}_0}\right) = S_0 \tilde D_2(f,g).
\end{equation}
We have
\begin{align}\label{E:ver3}
 \sigma\left(L_{\frac{1}{\nu^2} [f,g]_\ast} - \tilde L_{\frac{1}{\nu^2} [f,g]_{\tilde\ast}}\right) = 
\sigma\left(\frac{1}{\nu^2}[L_f,L_g] - \frac{1}{\nu^2}[\tilde L_f,\tilde L_g]\right) =\nonumber\\
\sigma\left(\frac{1}{\nu^2}([L_f,L_g] - [L_f, \tilde L_g]) +  \frac{1}{\nu^2}([L_f, \tilde L_g] - [\tilde L_f,\tilde L_g])\right) =\\
\sigma\left(\frac{1}{\nu}\left[ L_f, \frac{1}{\nu}(L_g - \tilde L_g)\right]\right) + \sigma\left(\frac{1}{\nu}\left[\frac{1}{\nu}(L_f - \tilde L_f), \tilde L_g\right]\right) = \nonumber\\
\{S_0f, S_1g\}_{\T} + \{S_1f, S_0g\}_{\T}. \nonumber
\end{align}
Combining (\ref{E:ver1} - \ref{E:ver3}) and observing that $D_2(f,g) - \tilde D_2(f,g) = 
\{f,g\}_1$ we obtain the proof of the first axiom in (\ref{E:bracket}):
\begin{align*}
 S_1 & \{f,g\}_0  =  \sigma\left(\frac{1}{\nu}\left(L_{\{f,g\}_0}   - \tilde L_{\{f,g\}_0}\right)\right) =\\
 \sigma & \left(L_{\frac{1}{\nu^2} [f,g]_\ast}  - \tilde L_{\frac{1}{\nu^2} [f,g]_{\tilde\ast}}\right) - S_0  D_2(f,g) + \tilde S_0 D_2(f,g) = \\
& \{S_0f, S_1g\}_{\T} + \{S_1f, S_0g\}_{\T} - S_0 \{f,g\}_1.
\end{align*}
Similarly one can check the second axiom in (\ref{E:bracket}). 

The following calculation  verifies the last axiom in (\ref{E:bracket}):
\begin{align*}
  \{S_0f, T_1g\}_{\T} = \sigma\left( \frac{1}{\nu}\left[ L_f, \frac{1}{\nu}(R_g - \tilde R_g)\right] \right) = - \sigma\left(\frac{1}{\nu^2}[L_f, \tilde R_g] \right) = \\
- \sigma\left( \frac{1}{\nu}\left[ \frac{1}{\nu}(L_f - \tilde L_f), \tilde R_g\right]\right) = -  \{S_1f, T_0g\}_{\T}.
\end{align*}
\end{proof}
\section{Deformation formal symplectic groupoid with separation of variables}

A K\"ahler-Poisson manifold is a complex manifold $M$ with a Poisson tensor of type $(1,1)$ with respect to the complex structure. In local holomorphic coordinates it is written as  $g^{lk}$ where the indices $k$ and $l$  are holomorphic and antiholomorphic, respectively.
The Jacobi identity for $g^{lk}$ takes the form
\begin{equation}\label{E:jac}
      g^{lk} \frac{\p g^{qp}}{\p z^k} = g^{qk} \frac{\p g^{lp}}{\p z^k} \mbox{ and } 
g^{lk} \frac{\p g^{qp}}{\p \bar z^l} = g^{lp} \frac{\p g^{qk}}{\p \bar z^l},
\end{equation}
where we assume summation over repeated indices. If the K\"ahler-Poisson tensor $g^{lk}$ is nondegenerate on $M$, then its inverse $g_{kl}$ is the metric tensor of a global pseudo-K\"ahler structure on $M$.

A star product $\ast$  on $M$ given by bidifferential operators $\{C_r\}$ is ``with separation of variables" if each operator $C_r$ differentiates its first argument only in antiholomorphic directions while the second argument only in holomorphic ones.  In particular,
\[
                     C_1(u,v) = g^{lk}\frac{\p u}{\p \bar z^l}\frac{\p v}{\p z^k}.
\]
If $a$ and $b$ are a local holomorphic and a local antiholomorphic function, respectively, then $a \ast f = af$ and $f \ast b = bf$ for any $f$. Otherwise speaking, the operators
\begin{equation}\label{E:laarbb}
               L_a = a \mbox{ and } R_b = b
\end{equation}
are pointwise multiplication operators. The star products with separation of variables on a pseudo-K\"ahler manifold are parameterized by the formal deformations of the pseudo-K\"ahler form (see \cite{CMP1}).  Assume that $M$ is a pseudo-K\"ahler manifold endowed with a pseudo-K\"ahler form $\omega_{-1}$. A formal deformation of the form $\omega_{-1}$ is a formal form
\[
    \omega = \frac{1}{\nu}\omega_{-1} + \omega_0 + \nu \omega_1 + \ldots
\]
such that $\omega_r$ are closed forms of type (1,1) with respect to the complex structure. The star product with separation of variables on $M$ parametrized by the form $\omega$ has the following property which completely characterizes it. Let $U$ be any contractible coordinate chart with complex cordinates $\{z^k, \bar z^l\}$. Then each form $\omega_r$ has a potential $\Phi_r$ on $U$, so that $\omega_r = i\p\bar \p \Phi_r$. Set
\[
      \Phi = \frac{1}{\nu}\Phi_{-1} + \Phi_0 + \nu \Phi_1 + \ldots.
\]
Then
\begin{equation}\label{E:sepvar}
     L_{\nu\frac{\p \Phi}{\p z^k}} = \nu\left(\frac{\p \Phi}{\p z^k} + \frac{\p}{\p z^k}\right) \mbox{ and } L_{\nu\frac{\p \Phi}{\p \bar z^l}} = \nu\left(\frac{\p \Phi}{\p \bar z^l} + \frac{\p}{\p \bar z^l}\right).
\end{equation}

Only a few examples of star products with separation of variables on general K\"ahler-Poisson manifolds are known (see, e.g., \cite{E}, \cite{LTW} ,\cite{CM}). 

It was proved in \cite{CMP3} that all star products with separation of variables on K\"ahler-Poisson manifolds are natural. Formulas (\ref{E:laarbb}) imply that the formal symplectic groupoid of a star product with separation of variables has the following property of separation of variables: for a local holomorphic function $a$ and antiholomorphic function $b$
\[
        Sa = a \mbox{ and } Tb = b.
\]
As shown in \cite{CMP3}, for each K\"ahler-Poisson manifold $(M,g^{lk})$ there exists a unique formal symplectic groupoid with separation of variables $G$ over $M$. There are simple formulas for its source and target mappings. We will use complex coordinates $\{z^k,\bar z^l\}$ on a neighborhood $U \subset M$ and the dual fiber coordinates 
$\{\zeta_k,\bar \zeta_l\}$ on $T^\ast U$. We introduce the operators
\[
     D^k = g^{lk} \frac{\p}{\p \bar z^l} \mbox{ and }  \bar D^l = g^{lk} \frac {\p}{\p z^k}.
\]
The Jacobi identity (\ref{E:jac}) is equivalent to the condition that $[D^k,D^p]=0$ and $[\bar D^l,\bar D^q] = 0$ for any $k,l,p,q$. Set
\[
    {\mathfrak D} = \zeta_k D^k \mbox{ and }  \bar{\mathfrak D} = \bar \zeta_l \bar D^l.
\]
The source and target mappings of the formal groupoid $G$ are given by the following local formulas,
\begin{equation}\label{E:sourcetarget}
      Sf = e^{\mathfrak D} f  \mbox{ and }  Tf = e^{\bar{\mathfrak D}} f
\end{equation}
for $f \in C^\infty(U)$. The exponentials in (\ref{E:sourcetarget}) are defined as formal series in the fiber variables $\{\zeta_k,\bar \zeta_l\}$.

An infinitesimal deformation $g^{lk} + \e h^{lk}$ of a K\"ahler-Poisson tensor $g^{lk}$ 
is given by a tensor $h^{lk}$ satisfying the equations
\begin{align}\label{E:jach}
     & g^{lk} \frac{\p h^{qp}}{\p z^k} + h^{lk} \frac{\p g^{qp}}{\p z^k}= g^{qk} \frac{\p h^{lp}}{\p z^k}  + h^{qk} \frac{\p g^{lp}}{\p z^k} \nonumber\\
\mbox{ and } & \\
& g^{lk} \frac{\p h^{qp}}{\p \bar z^l} + h^{lk} \frac{\p g^{qp}}{\p \bar z^l}= g^{lp} \frac{\p h^{qk}}{\p \bar z^l} + h^{lp} \frac{\p g^{qk}}{\p \bar z^l}. \nonumber
\end{align}
The following proposition can be proved along the same lines as Theorem 5 in \cite{CMP3}.
\begin{proposition}
 Given an infinitesimal deformation $g^{lk} + \e h^{lk}$ of a K\"ahler-Poisson tensor $g^{lk}$ on $M$, there exists a unique deformation formal symplectic groupoid with separation of variables over the manifold $(M, g^{lk} + \e h^{lk})$. Its source and target mappings are given by the formulas
\begin{equation}\label{E:sandt}
      Sf = e^{{\mathfrak D} + \e {\mathfrak E}} f  \mbox{ and }  Tf = 
e^{\bar{\mathfrak D} + \e \bar{\mathfrak E}} f,
\end{equation}
where
\[
     {\mathfrak E} = \zeta_k h^{lk} \frac{\p}{\p \bar z^l} \mbox{ and }  \bar{\mathfrak E} = \bar \zeta_l h^{lk} \frac {\p}{\p z^k}.
\]
\end{proposition}
Writing the source and target mappings (\ref{E:sandt}) in components, $S = S_0 + \e S_1$ and $T = T_0 + \e T_1$, we see that  $S_0$ and $T_0$ are the source and target mappings of the formal symplectic groupoid with separation of variables over $(M,g^{lk})$ given by the same formulas as in (\ref{E:sourcetarget}).
\begin{lemma} The components $S_1$ and $T_1$ of the source and target mappings (\ref{E:sandt}) are given by the following formulas:
\begin{align}\label{E:soneandtone}
     & S_1f = \left(\frac{e^{\ad ({\mathfrak D})} - 1}{ \ad ({\mathfrak D})} 
\left({\mathfrak E}\right) \right)S_0 f,  \nonumber\\
\mbox{ and } &\\
& T_1f = \left(\frac{e^{\ad (\bar {\mathfrak D})} - 1}{ \ad (\bar{\mathfrak D})} \left(\bar{\mathfrak E}\right)\right)T_0 f. \nonumber
\end{align}
\end{lemma}
\begin{proof} The right-hand sides of the formulas in (\ref{E:soneandtone}) are defined as formal series in the fiber variables $\{\zeta_k,\bar \zeta_l\}$. Consider the operator
\[
    u(t) = \exp\left(t({\mathfrak D} + \e {\mathfrak E})\right) = u_0(t) + \e u_1(t),
\]
where $u_0(t) = \exp(t {\mathfrak D})$. The component $u_1(t)$ satisfies the initial condition $u_1(0) = 0$ and the equation
\[
     \frac{du_1}{dt} = u_1 {\mathfrak D} + u_0 {\mathfrak E}.
\]
It can be checked directly that these conditions are satisfied by the operator
\[
  \left(\frac{e^{t\ad ({\mathfrak D})} - 1}{ \ad ({\mathfrak D})} \left({\mathfrak E}\right) \right)e^{t\mathfrak D}
\]
which, thus, equals $u_1(t)$. Setting $t=1$ we obtain the first formula in 
(\ref{E:soneandtone}). The second formula is proved similarly.
\end{proof}
Given the mappings $S_0,S_1$, one can recover the tensors $g^{lk}$ and $h^{lk}$ from the following formulas:
\begin{equation}\label{E:recov}
    E\left(\frac{\p}{\p\zeta_k}S_0 f\right) = g^{lk}\frac{\p f}{\p \bar z^l}, \ E\left(\frac{\p}{\p\zeta_k}S_1 f\right) = h^{lk}\frac{\p f}{\p \bar z^l}.
\end{equation}

Assume that $\varkappa$ is a closed (1,1)-form on $(M, g^{lk})$. It determines an infinitesimal deformation of the tensor $g^{lk}$ as follows. Let $\psi$ be a local potential of $\varkappa$, so that $\varkappa = i\p \bar \p \psi$. Set
\begin{equation}\label{E:defh}
    h^{lk} = - g^{lp} \frac{\p^2 \psi}{\p z^p \p \bar z^q} g^{qk}.
\end{equation}
One can check that $h^{lk}$ satisfies (\ref{E:jach}). On a pseudo-K\"ahler manifold, {\it every} deformation of $g^{kl}$ comes from such a closed (1,1)-form $\varkappa$. It turns out that for the infinitesimal deformation $g^{lk} + \e h^{lk}$ given by (\ref{E:defh}) formulas (\ref{E:soneandtone}) can be substantially simplified. 
\begin{lemma}
If $h^{lk}$ is given by formula (\ref{E:defh}) for some local function $\psi$, then 
\begin{eqnarray}\label{E:simpone}
S_1 f = \left(\frac{\p\psi}{\p z^p} - S_0 \frac{\p\psi}{\p z^p} \right) D^p S_0 f \mbox{ and } \nonumber 
\\
T_1 f = \left(\frac{\p\psi}{\p \bar z^q} - T_0 \frac{\p\psi}{\p \bar z^q} \right) \bar D^q T_0 f.
\end{eqnarray}
\end{lemma}
\begin{proof}
Observe that
\begin{equation}\label{E:expz}
    e^{\e \frac{\p\psi}{\p z^p} D^p} \left(e^{\mathfrak D}\right) e^{-\e \frac{\p\psi}{\p z^p} D^p} =
\exp\left( e^{\e \frac{\p\psi}{\p z^p} D^p} {\mathfrak D} e^{-\e \frac{\p\psi}{\p z^p} D^p} \right).
\end{equation}
Taking into account that $\e^2 = 0$, we rewrite (\ref{E:expz}) as follows:
\begin{equation}\label{E:explong}
   e^{\mathfrak D} + \e \left[ \frac{\p\psi}{\p z^p} D^p, e^{\mathfrak D} \right] = \exp \left({\mathfrak D} + \e \left[ \frac{\p\psi}{\p z^p} D^p,{\mathfrak D} \right] \right).
\end{equation}
Since $D^k$ and $D^p$ commute, we get that
\begin{equation}\label{E:explain}
     \left[ \frac{\p\psi}{\p z^p} D^p,{\mathfrak D} \right] = -\zeta_k g^{lp}\frac{\p^2 \psi}{\p z^p \p \bar z^q} g^{qk} \frac{\p}{\p \bar z^l} = \zeta_k h^{lk}\frac{\p}{\p \bar z^l} = {\mathfrak E}.
\end{equation}
Applying both sides of (\ref{E:explong}) to a function $f$ and taking (\ref{E:explain}) into account, we obtain that
\[
     S_0 f + \e \left[ \frac{\p\psi}{\p z^p} D^p, e^{\mathfrak D} \right] f = Sf,
\] 
whence, using that $D^p$ commutes with $e^{\mathfrak D}$, we arrive at a simple formula for $S_1$:
\begin{align}
    S_1 f = \left[ \frac{\p\psi}{\p z^p} D^p, e^{\mathfrak D} \right] f  = \frac{\p\psi}{\p z^p} D^p S_0 f - S_0\left(\frac{\p\psi}{\p z^p} D^p f\right) =\nonumber  \\
\left(\frac{\p\psi}{\p z^p} - S_0 \frac{\p\psi}{\p z^p} \right) D^p S_0 f.
\end{align}
One obtains the formula for $T_1$ in (\ref{E:simpone}) from a similar calculation.
\end{proof}
Since for each K\"ahler-Poisson manifold $M$ there is a unique formal symplectic groupoid with separation of variables over $M$, any two star products with separation of variables on $M$ have the same formal symplectic groupoid and one can consider the corresponding deformation groupoid. 
\begin{lemma}
The deformation groupoid of a pair of star products with separation of variables $(\ast, \tilde\ast)$ on a K\"ahler-Poisson manifold $M$ has the property of separation of variables: for a local holomorphic function $a$ and a local antiholomorphic function $b$,
\[
    Sa = a \mbox{ and } Tb =b.
\]
\end{lemma}
\begin{proof}
For a local holomorphic function $a$,
\[
     S_1a = \sigma\left(\frac{1}{\nu}(L_a - \tilde L_a)\right) =0,
\]
therefore, $Sa = S_0a + \e S_1 a = a$. Similarly, $Tb = b$ for a local antiholomorphic function $b$. 
\end{proof}
We will use formulas (\ref{E:simpone}) to calculate the deformation groupoid of a pair of star products with separation of variables on a pseudo-K\"ahler manifold $(M, \omega_{-1})$. 
\begin{proposition}
Let $\ast$ and $\tilde\ast$ be two star products with separation of variables on $M$ parametrized by the formal deformations 
\[
\omega = \frac{1}{\nu} \omega_{-1} + \omega_0 + \nu \omega_1 + \ldots \mbox{ and }
\tilde \omega = \frac{1}{\nu} \omega_{-1} + \tilde \omega_0 + \nu \tilde\omega_1 + \ldots
\]
of the pseudo-K\"ahler form $\omega_{-1}$. 
Then the corresponding deformation groupoid with separation of variables over $M$ corresponds to the deformation of the K\"ahler-Poisson structure by the closed (1,1)-form $\varkappa = \omega_0 - \tilde \omega_0$.
\end{proposition}
\begin{proof}
Consider local potentials 
\[
\Phi = \frac{1}{\nu} \Phi_{-1} + \Phi_0 + \nu \Phi_1 + \ldots \mbox{ and }
\tilde \Phi = \frac{1}{\nu} \Phi_{-1} + \tilde \Phi_0 + \nu \tilde\Phi_1 + \ldots
\]
of $\omega$ and $\tilde\omega$, respectively, on a contracible coordinate chart. Using ~(\ref{E:sepvar}), we get
\begin{align}\label{E:ontheone}
    S_1\left(\frac{\p\Phi_{-1}}{\p z^k}\right) = S_1\left(\nu\frac{\p\Phi}{\p z^k}\right) = \sigma\left(\frac{1}{\nu}\left(L_{\nu\frac{\p\Phi}{\p z^k}} - \tilde L_{\nu\frac{\p\Phi}{\p z^k}}  \right)\right) =\\
\sigma\left(\frac{1}{\nu}\left(L_{\nu\frac{\p\Phi}{\p z^k}} - \tilde L_{\nu\frac{\p\tilde\Phi}{\p z^k}}  \right) - \tilde L_{\frac{\p(\Phi-\tilde\Phi)}{\p z^k}} \right) = \frac{\p(\Phi_0 -\tilde\Phi_0)}{\p z^k} - S_0\frac{\p(\Phi_0 -\tilde\Phi_0)}{\p z^k}.\nonumber
\end{align}
Since
\[
    S_0 \frac{\p\Phi_{-1}}{\p z^k} = \frac{\p\Phi_{-1}}{\p z^k} + \zeta_k,
\]
we see that 
\begin{equation}\label{E:sobserv}
D^p S_0 \frac{\p\Phi_{-1}}{\p z^k} = \delta^p_k. 
\end{equation}
On the pseudo-K\"ahler manifold $M$ the mapping $S_1$ is given by the first formula in (\ref{E:simpone}) for some function $\psi$. Now, we get from (\ref{E:simpone}) and (\ref{E:sobserv}) that 
\begin{equation}\label{E:ontheother}
    S_1\left(\frac{\p\Phi_{-1}}{\p z^k}\right) = \frac{\p\psi}{\p z^k} - S_0 \frac{\p\psi}{\p z^k}.
\end{equation}
If we set $\varkappa = \omega_0 - \tilde \omega_0$ and take
\[
    \psi = \Phi_0 - \tilde \Phi_0,
\]
formulas (\ref{E:ontheone}) and (\ref{E:ontheother}) will agree. Since the K\"ahler-Poisson tensor $g^{lk}$ is nondegenerate, it can be seen, say, from formulas (\ref{E:recov}) that this can only happen if the deformation groupoid of the pair $(\ast,\tilde\ast)$ corresponds to the infinitesimal deformation of the Kahler-Poisson structure on $M$ determined by the form $\varkappa = \omega_0 - \tilde \omega_0$.
\end{proof}

\section{An alternative construction of a formal symplectic groupoid with separation of variables}

In this section we will work with an alternative, coordinate free construction of a formal symplectic groupoid with separation of variables introduced in \cite{CMP3}. It is based upon a semiclassical counterpart of a formal Berezin transform. We will apply this construction to deformation groupoids with separation of variables.

The formal Berezin transform $B$ of a star product with separation of variables $\ast$ on a K\"ahler-Poisson manifold $(M, g^{lk})$ is a formal differential operator on $M$ uniquely determined by the property that for a local holomorphic function $a$ and an antiholomorphic function $b$
\[
                   B(ab) = b \ast a.
\]
In particular,
\[
       B = 1 + \nu g^{lk} \frac{\p^2}{\p z^k \p \bar z^l} + \ldots.
\]
A star product with separation of variables is completely determined by its formal Berezin transform. It was proved in \cite{CMP3} that if the formal Berezin transform $B$ is written in the exponential form,
\[
      B = \exp \left\{\frac{1}{\nu} X\right\},
\]
then $X = \nu^2 X_2 + \nu^3 X_3 + \ldots$ is a natural operator with
\[
      X_2 = g^{lk} \frac{\p^2}{\p z^k \p \bar z^l}.
\]
It was shown in \cite{CMP3} that its $\sigma$-symbol $\sigma (X)$ is expressed in terms of the formal symplectic groupoid with separation of variables on $M$ and is even in the fiber variables on $\T$, i.e., $\s_{2k+1} (X_{2k+1}) =0$. It means that the order of the operator $X_{2k+1}$ is not greater than $2k$. Thus, the operator
\[
     Y: = \nu^2 X_3 + \nu^4 X_5 + \ldots
\]
is natural. It was the main motivation of this paper to understand the formal semiclassical origin of the $\sigma$-symbol $\sigma(Y)$ and to find an algorithm to compute it from formal semiclassical data. 

Recall the alternative description of all formal symplectic groupoids with separation of variables on a complex manifold $M$ from \cite{CMP3}. We will use local complex coordinates $\{z^k,\bar z^l\}$ on $M$ and the dual fiber coordinates $\{\zeta_k,\bar \zeta_l\}$ on $\T$. Given a function $K \in C^\infty(\T,Z)$, denote by $H_K$ the corresponding Hamiltonian vector field, so that for any $Q \in C^\infty(\T,Z)$,
\[
                       H_K Q = \{K,Q\}_{\T}.
\]
Let $\K$ denote the set of all elements $K \in C^\infty(\T,Z)$ such that the expansion of $K$ into homogeneous components starts at least with the quadratic term in the fiber variables, $K = K_2 + K_3 + \ldots$, and such that for any local holomorphic functions $a(z), \tilde a(z)$ and antiholomorphic functions $b(\bar z), \tilde b(\bar z)$ the following equations hold:
\begin{equation}\label{E:ehk}
     \{e^{H_K} a,\tilde a\}_{\T} = 0 \mbox{ and }  \{e^{H_K} b,\tilde b\}_{\T} = 0.
\end{equation}
Since the function $K$ starts with at least a quadratic term in the fiber variables, the operator $H_K$ raises the degree of filtration by the ideal $\I$ generated by $\zeta_k, \bar \zeta_l$ at least by one. Therefore, $\exp H_K$ is given by a series covergent in the $\I$-adic topology.
Equations (\ref{E:ehk}) can be rewritten in the recursive form:
\begin{align}\label{E:recurs}
     \{\{K_n,a\}_{\T},\tilde a\}_{\T} = - \sum_{k=2}^{n-1} \frac{1}{k!} \sum_{i_1 + \ldots i_k = n+k -1} \{H_{K_{i_1}}\ldots H_{K_{i_k}} a,\tilde a\}_{\T}, \nonumber\\
 \{\{K_n,b\}_{\T},\tilde b\}_{\T} = - \sum_{k=2}^{n-1} \frac{1}{k!} \sum_{i_1 + \ldots i_k = n+k -1} \{H_{K_{i_1}}\ldots H_{K_{i_k}} b,\tilde b\}_{\T}. 
\end{align}
As it was shown in \cite{CMP3}, there is a bijection of the set of all  global K\"ahler-Poisson structures on $M$ onto the set $\K$ which relates to each K\"ahler-Poisson tensor $g^{lk}$ on $M$ a unique element $K \in \K$ with
\begin{equation}\label{E:init}
             K_2 = g^{lk} \zeta_k \bar \zeta_l.
\end{equation}
The uniqueness part can be easily seen from (\ref{E:recurs}). It was also shown in ~\cite{CMP3} that every function $K \in \K$ is even in the fiber variables, 
\[
 K = K_2 + K_4 + \ldots.
\]
The components $K_r$ of the function $K\in \K$ corresponding to the tensor $g^{lk}$ are polynomials in partial derivatives of $g^{lk}$ and the fiber variables $\{\zeta_k,\bar \zeta_l\}$. The source and target mappings of the unique formal symplectic groupoid with separation of variables on a K\"ahler-Poisson manifold $(M, g^{lk})$ are expressed in terms of the corresponding element $K \in \K$ as follows: for a local holomorphic function $a(z)$ and antiholomorphic function $b(\bar z)$,
\begin{equation}\label{E:st}
     S(ab) = a \left(e^{H_K}b\right) \mbox{ and } T(ab) = \left(e^{H_K} a\right)b.
\end{equation}
Since the mappings $f \mapsto Sf$ and $f \mapsto Tf$ are given by formal series of differential operators whose coefficients are polynomials in the fiber variables $\{\zeta_k,\bar\zeta_l\}$, formulas (\ref{E:st}) determine them completely and uniquely.
Finally, if $B = \exp \frac{1}{\nu}X$ is the formal Berezin transform of an arbitrary star product with separation of variables on $(M, g^{lk})$, then
\[
                         K = \sigma(X).
\]

Now let $\ast$ and $\tilde \ast$ be two deformation quantizations with separation of variables on a K\"ahler-Poisson manifold $(M,g^{lk})$. They have the same formal symplectic groupoid with separation of variables over $(M,g^{lk})$ and therefore there exists the corresponding deformation formal symplectic groupoid. It was shown above that the deformation groupoid of a pair of star products with separation of variables  has the property of separation of variables.

Let $\K[\e]$ denote the set of all elements $F = K + \e J \in C^\infty(\T,Z)[\e]$ such that the expansion of $F$ into homogeneous components starts at least with the quadratic term in the fiber variables, 
\[
F = F_2 + F_3 + \ldots
\]
and such that for any local holomorphic functions $a(z), \tilde a(z)$ and antiholomorphic functions $b(\bar z), \tilde b(\bar z)$ the following equations hold:
\begin{equation}\label{E:fpert}
  \{e^{H_F} a,\tilde a\}_{\T} = 0 \mbox{ and }  \{e^{H_F} b,\tilde b\}_{\T} = 0.
\end{equation}
Equations (\ref{E:fpert}) can be written in the recursive form analogous to 
(\ref{E:recurs}). One can prove along the same lines the following theorem.
\begin{theorem}
There is a bijection of the set of all infinitesimal deformations $g^{lk} + \e h^{lk}$ of K\"ahler-Poisson tensors on $M$ onto $\K[\e]$ which relates to each infinitesimal deformation $g^{lk} + \e h^{lk}$ a unique element $F = K + \e J \in \K[\e]$ such that
\begin{equation}\label{E:inits}
      K_2 = g^{lk} \zeta_k \bar \zeta_l \mbox{ and } J_2 = h^{lk} \zeta_k \bar \zeta_l.
\end{equation}
The function $F$ is even in the fiber variables, so that 
\[
K = K_2 + K_4 + \ldots \mbox{ and } J = J_2 + J_4 + \ldots.
\] 
Moreover, for each infinitesimal deformation $g^{lk} + \e h^{lk}$ on $M$ there exists a unique deformation formal symplectic groupoid with separation of variables over $(M, g^{lk} + \e h^{lk})$. The source and target mappings of the deformation groupoid, $S = S_0 + \e S_1$ and $T = T_0 + \e T_1$, are given by the equations 
\begin{equation}\label{E:stpert}
    S(ab) = a \left(e^{H_F}b\right) \mbox{ and } T(ab) = \left(e^{H_F} a\right)b,
\end{equation}
where $F$ is the element of $\K[\e]$ corresponding to $g^{lk} + \e h^{lk}$. The components $S_0$ and $T_0$ are the source and target mappings of the unique formal symplectic groupoid over $(M,g^{lk})$ and $K$ is the element of $\K$ corresponding to $g^{lk}$ so that
\[
      S_0 (ab) = a\left(e^{H_K}b \right) \mbox{ and } T_0(ab) = \left(e^{H_K} a\right)b.
\]
\end{theorem}
Given a function $f \in C^\infty(M)$, one can easily verify the formula
\begin{align}\label{E:var}
     \frac{d}{d\tau}\bigg |_{\tau =0}  e^{\left(H_K + \tau H_J\right)}f = \left(\frac{e^{\ad H_K} -1}{\ad H_K} H_J\right) e^{H_K}f = \\
\left\{\frac{e^{H_K} - 1}{H_K}J, e^{H_K}f \right\}.\nonumber
\end{align}
 Formulas (\ref{E:stpert}) and (\ref{E:var}) imply that
\begin{align*}
   S_1 (ab) = & a \left\{\frac{e^{H_K} - 1}{H_K}J, S_0 b \right \}_{\T} \mbox{ and }\\
 & T_1(ab) = \left\{\frac{e^{H_K} - 1}{H_K}J, T_0 a \right \}_{\T} b.
\end{align*}

\section{On the logarithm of the formal Berezin transform}

Let $\ast$ and $\tilde\ast$ be two deformation quantizations with separation of variables on a K\"ahler-Poisson manifold $(M,g^{lk})$. Write their corresponding formal Berezin transforms in the exponential form:
\[
     B = \exp\left \{ \frac{1}{\nu} X \right\} \mbox{ and } \tilde B = \exp\left \{ \frac{1}{\nu}\tilde X \right\}.
\]
It was shown in \cite{CMP3} that
\[
     \sigma(X) = \sigma(\tilde X) = K
\]
is the element of $\K$ corresponding to $g^{lk}$. Thus, $X - \tilde X \in \nu \N$. Set
\begin{equation}\label{E:funj}
      J = \sigma\left( \frac{1}{\nu}(X - \tilde X) \right) \mbox{ and } F = K + \e J.
\end{equation}
We want to show that $F\in \K[\e]$ and that the deformation groupoid of the pair of star products $\ast$ and $\tilde\ast$ is over $(M, g^{lk} + \e h^{lk})$, where  $g^{lk} + \e h^{lk}$ is the infinitesimal deformation of $g^{lk}$ corresponding to $F$.

\begin{proposition}
The function $F = K + \e J$ defined in (\ref{E:funj}) is an element of $\K[\e]$.
\end{proposition}
\begin{proof}
Given a function $f \in C^\infty(M)$, one can check that the operators
\[
      e^{t\left(\frac{1}{\nu}X\right)}fe^{-t\left(\frac{1}{\nu}X\right)} = e^{t \ad \left(\frac{1}{\nu}X\right)} f \mbox{ and } e^{t\left(\frac{1}{\nu}\tilde X\right)}fe^{-t\left(\frac{1}{\nu}\tilde X\right)} = e^{t \ad \left(\frac{1}{\nu}\tilde X\right)} f 
\]
are natural and have the same $\sigma$-symbol $e^{t H_K}f$.
The operator
\[
      U_t(f) = \frac{1}{\nu}\left(e^{t \ad \left(\frac{1}{\nu}X\right)} f -e^{t \ad \left(\frac{1}{\nu}\tilde X\right)} f \right)
\]
is thus also natural. Set
\[
    u = \sigma \left(U_t(f)\right).
\]
We have $u|_{t=0}=0$ and
\begin{align*}
  \frac{du}{dt} =  \sigma &\left( \frac{1}{\nu} \left[ \frac{1}{\nu}X,e^{t \ad \left(\frac{1}{\nu}X\right)}f \right] - \frac{1}{\nu}\left[ \frac{1}{\nu}\tilde X,e^{t \ad \left(\frac{1}{\nu}\tilde X\right)}f \right] \right) =\\
 \sigma &\left( \frac{1}{\nu} \left[ \frac{1}{\nu}(X - \tilde X),e^{t \ad \left(\frac{1}{\nu}X\right)}f \right] \right) +\\
  \sigma &\left( \frac{1}{\nu} \left[\tilde X,\frac{1}{\nu}\left(e^{t \ad \left(\frac{1}{\nu}X\right)}f - e^{t \ad \left(\frac{1}{\nu}\tilde X\right)}f \right) \right]\right)= \\
 & \left\{J, e^{t H_K}f \right\}_{\T} + \{K,u\}_{\T} = H_J e^{t H_K}f + H_K u.
\end{align*}
Set
\[
    v = e^{-tH_K} u.
\]
Then $v|_{t=0}=0$ and
\begin{align*}
  \frac{dv}{dt} = e^{-tH_K} \frac{du}{dt} - e^{-tH_K} H_K u = e^{-tH_K}\left(H_J e^{t H_K}f + H_K u\right)- e^{-tH_K} H_K u =\\
\left(e^{-tH_K} H_J e^{tH_K}\right) f=  \left(e^{-t\ad H_K} H_J\right) f = \left\{ e^{-tH_K} J, f \right\}_{\T}.
\end{align*}
It follows that
\[
     v =  \left\{ \frac{1 -e^{-tH_K}}{H_K}J, f \right\}_{\T},
\]
whence we derive that
\begin{equation}\label{E:u}
  u = \left\{ \frac{e^{tH_K} - 1}{H_K}J, e^{tH_K} f \right\}_{\T}.
\end{equation}
Let $a$ and $b$ be a local holomorphic and a local antiholomorphic function on $M$, respectively. It was proved in \cite{CMP3} that 
\begin{equation}\label{E:lbra}
    L_b = B b B^{-1}, \ R_a = B a B^{-1}, \ \tilde L_b = \tilde B b \tilde B^{-1}, \ \tilde R_a = \tilde B a \tilde B^{-1}.
\end{equation}
 Setting $t=1$ and $f = b$ in (\ref{E:u}) we obtain from formulas (\ref{E:sonetonestar}) and (\ref{E:lbra}) and the definition of the function $u$  that
\begin{equation}\label{E:sone}
     S_1 b = u_1(b) =  \left\{ \frac{e^{H_K} - 1}{H_K}J, e^{H_K} b \right\}_{\T} = 
\left\{ \frac{e^{H_K} - 1}{H_K}J, S_0 b \right\}_{\T},
\end{equation}
where $S_0$ and $S_1$ are the components of the source mapping of the deformation groupoid of the pair $(\ast,\tilde\ast)$. Setting $t=1$ and $f=a$ in (\ref{E:u}) we get
\[
   T_1 a = \left\{ \frac{e^{H_K} - 1}{H_K}J, T_0 a \right\}_{\T},
\]
where $T_0$ and $T_1$ are the components of the target mapping of that groupoid. Using formulas (\ref{E:var}) and (\ref{E:sone}), we obtain that for the function $F = K + \e J$,
\[
     e^{H_F} b = e^{H_K + \e H_J} b = S_0 b + \e S_1 b = Sb.
\]
Similarly,
\[
   e^{H_F} a = Ta.
\]
Since $Sf$ and $Tg$ Poisson commute for any $f,g$,  this implies that for any local holomorphic functions $a, \tilde a$ and local antiholomorphic functions $b,\tilde b$,
\[
    \left\{ e^{H_F} a, \tilde a\right\} = \{Ta, S\tilde a\} = 0 \mbox{ and } \left\{ e^{H_F} b, \tilde b\right\} = \{Sb, T\tilde b\} = 0.
\]
Thus we have shown that $F \in \K[\e]$. 
\end{proof}
It follows from the definition (\ref{E:funj}) of the function $J$ that
\[
    J_2 = h^{lk} \zeta_k \bar \zeta_l = \s_2(X_3 - \tilde X_3).
\]
For a local holomorphic function $a$ and a local antiholomorphic function $b$ we have that $Ba =a, Bb = b$ and therefore $Xa =Xb =0$. For the same reason, $\tilde Xa = \tilde X b = 0$. It  follows that
\begin{equation}\label{E:xx}
     X_3 - \tilde X_3 = h^{lk} \frac{\p^2}{\p z^k \p \bar z^l}.
\end{equation}
Starting with formula (\ref{E:xx}), one can infer that the deformation groupoid of the pair $(\ast,\tilde\ast)$ is the unique deformation formal symplectic groupoid with separation of variables over $(M, g^{lk} + \e h^{lk})$. Also, the function $F = K + \e J$ can be recursively calculated from equations (\ref{E:fpert}).

Finally, given a star product $\ast$ with separation of variables on a K\"ahler-Poisson manifold $(M,g^{lk})$, we want to determine the $\sigma$-symbol of the natural operator
\[  
     Y = \nu^2 X_3 + \nu^4 X_5 + \ldots
\]
obtained from the Berezin transform $B = \exp\frac{1}{\nu}X$ of that star product. It was shown in \cite{CMP3} that the dual star product $\tilde \ast$ given by the formula
\[
      u \tilde\ast v = B^{-1} (Bv \ast Bu)
\]
is a star product with separation of variables on the 
K\"ahler-Poisson manifold $(M, -g^{lk})$. 
Its formal Berezin transform $\tilde B$ is the inverse of $B$, so that \[
      \tilde B = B^{-1} = \exp\left\{-\frac{1}{\nu}X \right\} = \exp\{-\nu X_2 - \nu^2 X_3 - \ldots\}.
\]
Now we want to replace the formal parameter $\nu$ in the dual star product $\tilde\ast$ with $-\nu$. The resulting star product $\hat\ast$ is again a star product with separation of variables on the K\"ahler-Poisson manifold $(M, g^{lk})$. Its formal Berezin transform $\hat B = \exp\frac{1}{\nu}\hat X$ is such that $\hat X_k = (-1)^k X_k$, i.e.,
\[
  \hat X = \nu^2 X_2 - \nu^3 X_3 - \nu^4 X_4 + \ldots.
\]
Therefore, 
\[
    Y = \frac{1}{2\nu}(X - \hat X).
\]
Now we see from formula (\ref{E:xx}) that the $\sigma$-symbol of $Y$ can be found from the deformation groupoid of the pair $(\ast,\hat\ast)$ which is the deformation groupoid with separation of variables over $(M, g^{lk} +\e h^{lk})$, where the tensor $h^{kl}$ is expressed in terms of the operator $X_3$ as follows:
\begin{equation}\label{E:xthree}
     X_3 = \frac{1}{2}(X_3 - \hat X_3) = \frac{1}{2} h^{lk} 
\frac{\p^2}{\p z^k \p \bar z^l}.
\end{equation}
Let $F = K+\e J$ be the element of $\K[\e]$ corresponding to $g^{lk} +\e h^{lk}$. Then, finally,
\begin{equation}\label{E:yj}
      \sigma(Y) = \frac{1}{2}J.
\end{equation}
We have proved the following theorem.
\begin{theorem}
Given a star product with separation of variables on a K\"ahler-Poisson manifold $(M, g^{lk})$ with the formal Berezin transform $B = \exp \left(\frac{1}{\nu}X\right)$,
where $X = \nu^2 X_2 + \nu^3 X_3 + \ldots$, then the formal differential operator
\[
   Y = \nu^2 X_3 + \nu^4 X_5 +\ldots
\] 
is natural, the operator $X_3$ is given by formula (\ref{E:xthree}) for some (1,1)-tensor $h^{lk}$, $g^{lk} + \e h^{lk}$ is an infinitesimal deformation of the K\"ahler-Poisson tensor $g^{lk}$, and the function $F = K + \e J \in \K[\e]$ corresponding to the deformation formal symplectic groupoid with separation of variables over $\left(M,g^{lk} + \e h^{lk}\right)$ is such that formula (\ref{E:yj}) holds.
\end{theorem}

\end{document}